\documentclass[a4paper,10pt]{article}
\usepackage{amssymb}
\usepackage{amsmath}
\usepackage{amsthm}
\usepackage{amsfonts}
\usepackage{hyperref}

\usepackage[english]{babel}
\newtheorem{theorem}{Theorem}[section]

\newtheorem{lemma}{Lemma}[section]

\newtheorem{corollary}{Corollary}[section]

\newtheorem{remark}{Remark}[section]

\numberwithin{equation}{section}
\newcommand{\dom}    {\mathrm{dom}\,}

\def\sS{{\mathfrak S}}

      \def\dC{{\mathbb C}}

   \def\dN{{\mathbb N}}   
      \def\dR{{\mathbb R}}

\def\cS{{\mathcal S}}

\author{Vladimir Lotoreichik}

\title{
Singular continuous spectrum of half-line Schr\"odinger operators with point interactions on a sparse set
}
\date{}

\begin{document}

\maketitle
\begin{abstract}
\noindent
We say that a discrete set  $X =\{x_n\}_{n\in\dN_0}$ on the half-line
$$0=x_0 < x_1 <x_2 <x_3<\dots <x_n<\dots <+\infty$$ 
is sparse if the distances $\Delta x_n = x_{n+1} -x_n$ between neighbouring points satisfy the condition $\frac{\Delta x_{n}}{\Delta x_{n-1}} \rightarrow +\infty$. In this paper half-line Schr\"odinger operators with point $\delta$- and $\delta^\prime$-interactions on a sparse set
are considered. Assuming that strengths of point interactions tend to $\infty$ we give simple sufficient conditions for such Schr\"odinger operators to have non-empty singular continuous spectrum and to have purely singular continuous spectrum, which coincides with  $\dR_+$. \\
\end{abstract}

%----------------------------------------------------------------------
%
\section{Introduction}
\label{sec:intro}
%
%----------------------------------------------------------------------
One-dimensional Schr\"odinger operators with $\delta$-interactions on a discrete set describe the behaviour of a non-relativistic charged particle in a one-dimensional lattice. Periodic models of such a type were considered first by Kronig and Penney in~\cite{KP31}.  The classical results and a detailed list of references on the theory of one-dimensional Schr\"odinger operators with $\delta$- and $\delta^\prime$-interactions on a discrete set can be found in the monograph~\cite{AGHHE05}. Schr\"odinger operators with point interactions are considered, for instance, in~\cite{AKM10,BL10,B85,BSW95,CS94, GK85,GO10,K89,KM10, M95,M96,N03,SS99} and in many other works. Our list of references is far from being complete, although various recent significant works are mentioned.
 
In the present paper we are interested in the effect first discovered by Pearson in~\cite{P78} for one-dimensional Schr\"odinger operators with regular sparse potentials. Sparse potentials were also discussed by Gordon, Molchanov and Zagany in~\cite{GMoZ91}, where some results were given without proofs. Under some assumptions on the degree of sparseness of the potential one gets purely singular continuous spectrum. 
An example of such a potential was constructed by Simon and Stolz in~\cite{SS96}. According to the results of~\cite{SS96} the half-line Schr\"odinger operator 
\begin{equation*}
-\frac{d^2}{dx^2} +V
\end{equation*}
with the potential 
\begin{equation}
\label{eq:V}
V(x) = \begin{cases} n,\quad \text{if}\quad \bigl|x - e^{2n^{3/2}}\bigr| < \frac12,\\
0,\quad \text{otherwise,}
\end{cases}
\end{equation}
and an arbitrary self-adjoint boundary condition at the origin has the following structure of the spectrum
\begin{equation*}
\sigma_{\rm p} = \sigma_{\rm ac} = \varnothing\quad\text{and}\quad \quad \sigma_{\rm sc} = \bigl[0,+\infty\bigr).
\end{equation*}
The main achievement of this construction is the stability of the singular continuous spectrum under ''small'' variations of the potential $V$ in~\eqref{eq:V} and arbitrary self-adjoint variations of the boundary condition at the origin. This situation is non-typical for other known examples with singular continuous spectrum.

Recently sparse potentials have attracted the attention again \cite{B07,BF09,EL10,T05}. In particular, Breuer and Frank established in~\cite{BF09} sufficient conditions for the spectrum of the Laplace operator on a metric sparse tree to be purely singular continuous. 

In the present paper we establish the existence of such an effect for Schr\"odinger operators with point  $\delta$- and $\delta^\prime$-interactions on a sparse discrete set.  As in the classical case the obtained singular continuous spectrum is stable under ''small'' variations of the discrete set and the strengths of interactions.

Let $\alpha = \{\alpha_n\}_{n\in\dN}$ be a sequence of real numbers. Let $X=\{x_n\}_{n\in\dN_0}$  be a discrete set on the half-line 
\begin{equation*}
\begin{split}
&0=x_0 < x_1 < x_2 < x_3 < \dots < x_n<\dots <+\infty\\
\end{split}
\end{equation*}
such that the sequence $\Delta x_n = x_{n+1}-x_n$  satisfies  the condition
\begin{equation}
\label{eq:cond}
\inf_{n\in\dN_0} \Delta x_n >0. 
\end{equation}
We consider half-line Schr\"odinger operators $H_{\delta, X,\alpha}$,  and $H_{\delta^\prime, X, \alpha}$ formally given by expressions
\begin{equation}
\label{eq:HH'}
H_{\delta, X,\alpha} = -\frac{d^2}{dx^2} + \sum_{n\in\dN} \alpha_n \delta_{x_n}\quad\text{and}\quad 
H_{\delta^\prime, X,\alpha} = -\frac{d^2}{dx^2} + \sum_{n\in\dN} \alpha_n \langle\delta_{x_n}^\prime,\cdot\rangle\delta^\prime_{x_n},
\end{equation}
where $\delta_{x}$ is the delta distribution supported by the point $x\in\dR_+$ and $\delta_{x}^\prime$ is its derivative. The corresponding operators $H_{\delta,X,\alpha}$ and $H_{\delta^\prime,X,\alpha}$ turn out to be self-adjoint in $L^2(\dR_+)$.

A discrete set $X$ is said to be sparse in the case the following condition holds
\begin{equation}
\label{eq:sparse}
\frac{\Delta x_{n}}{\Delta x_{n-1}} \rightarrow +\infty.
\end{equation} 
Our main results are contained in the following theorem.
\begin{theorem}
\label{thm:main}
Let $X = \{x_n\}_{n\in\dN_0}$ be a discrete set on the half-line such that~\eqref{eq:sparse} holds. Let $\alpha= \{\alpha_n\}_{n\in\dN}$ be a sequence of real numbers such that $\alpha_n \rightarrow\infty$. Let $H_{\delta,X,\alpha}$ and $H_{\delta^\prime,X,\alpha}$ be self-adjoint half-line Schr\"odinger operators as in~\eqref{eq:HH'}, strictly defined in \eqref{eq:opdelta} and in \eqref{eq:opdelta'}. Define $a\in\dR_+\cup\{+\infty\}$ by the limit
\begin{equation}
a := \liminf_{n\rightarrow \infty} \frac{\Delta x_{n}}{\Delta x_{n-1}\alpha_n^2}.
\end{equation}
Then the following assertions hold:
\begin{itemize}
\item[(i)] if $0 < a <+\infty$, then
\begin{itemize}
\item[(a)] the spectrum of the operator $H_{\delta,X,\alpha}$ has the following structure:
\begin{itemize}
\item[(pp)] $\sigma_{\rm pp}\cap \dR_+ \subseteq \bigl[0,1/a\bigr]$,
\item[(sc)] $\bigl[1/a,+\infty\bigr)\subseteq\sigma_{\rm sc} \subseteq \bigl[0,+\infty\bigr)$,
\item[(ac)] $\sigma_{\rm ac}= \varnothing$;
\end{itemize}
\item[(b)] the spectrum of the operator $H_{\delta^\prime,X,\alpha}$ has the structure:
\begin{itemize}
\item[(pp)] $\sigma_{\rm pp}\cap\dR_+ \subseteq \bigl[a,+\infty\bigr)$,
\item[(sc)] $\bigl[0,a\bigr]\subseteq\sigma_{\rm sc} \subseteq \bigl[0,+\infty\bigr)$,
\item[(ac)] $\sigma_{\rm ac} = \varnothing$.
\end{itemize}
\end{itemize}
In particular, in this case the singular continuous spectrum of both operators $H_{\delta,X,\alpha}$ and $H_{\delta^\prime,X,\alpha}$ is non-empty;
\item[(ii)] if $a = +\infty$ and if the sequence $\alpha$ contains only positive real numbers, then the spectrum of both operators $H_{\delta,X,\alpha}$ and $H_{\delta^\prime,X,\alpha}$ is purely singular continuous and coincides with $\dR_+$.
\end{itemize}
\end{theorem}

As an example the spectrum of the Schr\"odinger operator formally given by the expression
\begin{equation*}
-\frac{d^2}{dx^2} + \sum_{n\in\dN} n^{1/4}\delta_{n!}
\end{equation*}
is purely singular continuous and coincides with $\dR_+$, owing to $a=+\infty$.
In another example the singular continuous spectrum of the Schr\"odinger operator formally given by the expression
\begin{equation*}
-\frac{d^2}{dx^2} + \sum_{n\in\dN} n^{1/2}\delta_{n!}
\end{equation*}
is non-empty and contains the interval $\bigl[1,+\infty)$, owing to $a=1$. 

{\bf Notations:} By $\dN_0$ we denote $\dN \cup\{0\}$. We write $T\in\sS_\infty$ in the case the operator $T$ is compact. By $\sigma_{\rm p}, \sigma_{\rm pp}, \sigma_{\rm ac}$ and $\sigma_{\rm sc}$ we denote point, pure point, absolutely continuous and singular continuous spectra. By $\sigma_{\rm ess}$ we denote the essential spectrum. 
We write $\psi \in {\rm AC}_{\rm loc}(I)$ in the case the function $\psi$ is locally absolutely continuous on the set $I$.   
\subsection*{Definitions of operators with point interactions}
%---------------------------------------------------------------------------------
%---------------------------------------------------------------------------------
We give strict definitions of operators $H_{\delta,X,\alpha}$ and $H_{\delta^\prime,X,\alpha}$ using the language of boundary conditions.

Let $\alpha = \{\alpha_n\}_{n\in\dN}$ be a sequence of real numbers. Let $X = \{x_n\}_{n\in\dN_0}$ be a discrete set of points on the half-line arranged in increasing order such that the sequence $\Delta x_n = x_{n+1}-x_n$ satisfies the condition~\eqref{eq:cond}. 
  
Let us introduce two classes of functions on the half-line 
\begin{equation}
\label{eq:delta}
\cS_{\delta,X,\alpha} =\Bigl\{ \psi\colon \psi,\psi^\prime\in {\rm AC}_{\rm loc}(\dR_+\setminus X)\colon\psi(0) = 0,~\begin{smallmatrix} \psi(x_n+)  = \psi(x_n-) = \psi(x_n) \\ \psi^\prime(x_n+) - \psi^\prime(x_n-) = \alpha_n\psi(x_n)\end{smallmatrix}\Bigr\}
\end{equation}
and
\begin{equation}
\label{eq:delta'}
\cS_{\delta^\prime,X,\alpha} =\Bigl\{ \psi\colon\psi,\psi^\prime \in {\rm AC}_{\rm loc}(\dR_+\setminus X)\colon \psi(0) = 0,~ \begin{smallmatrix} \psi^\prime(x_n+)  = \psi^\prime(x_n-) = \psi^\prime(x_n) \\ \psi(x_n+) - \psi(x_n-) = \alpha_n \psi^\prime(x_n)\end{smallmatrix}\Bigr\}.
\end{equation}
The operator $H_{\delta,X,\alpha}$ is defined in the following way
\begin{equation}
\label{eq:opdelta}
H_{\delta,X,\alpha}\psi = -\psi^{\prime\prime},\quad 
\dom H_{\delta,X,\alpha} = \bigl\{ \psi\in L^2(\dR_+)\cap \cS_{\delta,X,\alpha}\colon -\psi^{\prime\prime}\in L^2(\dR_+)\bigr\}, 
\end{equation}
and the operator $H_{\delta^\prime,X,\alpha}$ is defined analogously
\begin{equation}
\label{eq:opdelta'}
H_{\delta^\prime,X,\alpha}\psi = -\psi^{\prime\prime},\quad 
\dom H_{\delta^\prime,X,\alpha} = \bigl\{ \psi\in L^2(\dR_+)\cap \cS_{\delta^\prime,X,\alpha}\colon -\psi^{\prime\prime}\in L^2(\dR_+)\bigr\}. 
\end{equation}
Both operators $H_{\delta,X,\alpha}$ and $H_{\delta^\prime,X,\alpha}$ are 
self-adjoint in  $L^2(\dR_+)$, according to~\cite{GK85,K89}.

%---------------------------------------------------------------------------------
%----------------------------------------------------------------------------------

%-------------------------------------------------------------------------------
%-------------------------------------------------------------------------------
\section{Absence of  point spectrum on a subinterval of $\dR_+$} 
%-------------------------------------------------------------------------------
%-------------------------------------------------------------------------------
\label{sec:3}
In this section we establish sufficient conditions on $X$ and $\alpha$, which give operators $H_{\delta,X,\alpha}$ and $H_{\delta^\prime,X,\alpha}$ with absence of point spectra on a subinterval of $\dR_+$.  We adapt the approach suggested for regular potentials by Simon and Stolz in~\cite{SS96}  to the case of point interactions. Similar idea has been used recently by Breuer and Frank~\cite{BF09} in order to prove the absence of point spectrum of the Laplace operator on a sparse metric tree from a certain class.

Let us consider a function $\psi$ such that for some $\lambda >0$
\begin{equation}
\label{eq:diff}
-\psi^{\prime\prime}(x) = \lambda\psi(x)\quad \text{for all} \quad x\in\dR_+\setminus X.
\end{equation}  
We give assumptions on $X$ and $\alpha$ such that any function $\psi\in L^2(\dR_+)\cap \cS_{\delta,X,\alpha}$, satisfying~\eqref{eq:diff} for any $\lambda$ from a subinterval $I \subseteq\dR_+$, is trivial. As the result under these assumptions the point spectrum of the operator $H_{\delta,X,\alpha}$ is absent on the interval $I$.  We give analogous assumptions on $X$ and $\alpha$ in the $\delta^\prime$-case.

We need the following subsidiary lemma.
\begin{lemma}
\label{lem0}
Let $X = \{x_n\}_{n\in\dN_0}$ be a discrete set on the half-line such that~\eqref{eq:cond} holds. Let a function $\psi$ be such that $\psi,\psi^\prime\in {\rm AC}_{\rm loc}(\dR_+\setminus X)$. Assume that~$\psi$ satisfies~\eqref{eq:diff} for some $\lambda >0$.  If the sequence of vectors $\xi_n :=\begin{pmatrix} \psi(x_n+) \\ \psi^\prime(x_n+)\end{pmatrix}$ satisfies the condition
\begin{equation}
\label{eq:infty}
\sum_{n=0}^\infty \Delta x_n\|\xi_n\|^2_{\dC^2}  =\infty,
\end{equation}
then
\begin{equation*}
\int_0^\infty |\psi(x)|^2dx =\infty.
\end{equation*}
\end{lemma}
\begin{proof}
For a point $x\in(x_n,x_{n+1})$ the following connection between $\xi_n$ and $\bigl(\begin{smallmatrix} \psi(x) \\ \psi^\prime(x)\end{smallmatrix}\bigr)$ holds
\begin{equation}
\label{eq:xix}
\xi_n = M_\lambda(x_n-x)\begin{pmatrix} \psi(x) \\ \psi^\prime(x)\end{pmatrix},
\end{equation}
where the matrix $M_\lambda(d)$ is the fundamental matrix, having the following explicit form
\begin{equation}
\label{eq:md}
M_\lambda(d) = \begin{pmatrix} \cos(d\sqrt\lambda ) & \frac{\sin(d\sqrt\lambda)}{\sqrt\lambda}\\
-\sqrt\lambda\sin(d\sqrt\lambda) &\cos(d\sqrt\lambda)
\end{pmatrix}.
\end{equation}
It follows from the identity~\eqref{eq:xix} that
\begin{equation*}
\Biggl\|\begin{pmatrix} \psi(x) \\ \psi^\prime(x)\end{pmatrix}\Biggr\|_{\dC^2} \ge \frac{ \|\xi_n\|_{\dC^2}}{\|M_\lambda(x-x_n)\|}.
\end{equation*}
The norm of the fundamental matrix $M_\lambda(d)$  is bounded as a function of $d$, namely
\begin{equation}
\sup_{d\in\dR}\|M_\lambda(d)\|\le C_\lambda <+\infty.
\end{equation}
According to~\eqref{eq:infty} we have
\begin{equation}
\label{eq:eq1}
\int_0^\infty\bigl\|\bigl(\begin{smallmatrix} \psi(x) \\ \psi^\prime(x)\end{smallmatrix}\bigr)\bigr\|^2_{\dC^2}dx \ge \frac{1}{(C_\lambda)^2}\sum_{n=0}^\infty\int_{x_n}^{x_{n+1}}\|\xi_n\|_{\dC^2}^2dx = \frac{1}{(C_\lambda)^2}\sum_{n=0}^\infty\Delta x_n\|\xi_n\|^2_{\dC^2}  =\infty.
\end{equation}
If $\psi \in L^2(\dR_+)$, then $\psi^{\prime\prime}=-\lambda\psi\in L^2(\dR_+)$. Taking into account the inequality 
\begin{equation}
\|\psi^\prime\|_{L^2(\dR_+)}^2\le a\|\psi\|_{L^2(\dR_+)}^2 + b\|\psi^{\prime\prime}\|_{L^2(\dR_+)}^2,
\end{equation}
(see, e.~g.,~\cite[\S III.10]{EE}), which holds for some constants $a,b>0$, we get $\psi^\prime\in L^2(\dR_+)$. 
Finally, we come to the conclusion that for the divergence of the integral on the left hand side in~\eqref{eq:eq1} we need $\psi\notin L^2(\dR_+)$.
\end{proof}

Let $\alpha =\{\alpha_n\}_{n\in\dN}$ be a sequence of real numbers, then we introduce for all $\lambda>0$ the sequence  $\{A_n(\lambda)\}_{n\in\dN_0}$ 
\begin{equation}
\label{eq:AB}
A_n(\lambda) :=  \prod_{i=1}^n \Bigl(1 + \frac{|\alpha_i|}{\sqrt{\lambda}}\Bigr).
\end{equation}

Further we need two lemmas, which give asymptotic estimates from below of the behaviour of functions in the classes  $\cS_{\delta,X,\alpha}$ and $\cS_{\delta^\prime,X,\alpha}$, satisfying~\eqref{eq:diff} for some $\lambda >0$. 
\begin{lemma}
\label{lem1}
Let $X = \{x_n\}_{n\in\dN_0}$ be a discrete set on the half-line such that~\eqref{eq:cond} holds.
Let $\alpha = \{\alpha_n\}_{n\in\dN}$ be a sequence of real numbers.
Let a function $\psi \in \cS_{\delta,X,\alpha}$ be such~\eqref{eq:diff} holds for some $\lambda>0$. Let the sequence $\{A_n(\lambda)\}_{n\in\dN_0}$ be defined as in~\eqref{eq:AB}.
Then norms of vectors $\xi_n :=\begin{pmatrix} \psi(x_n+) \\ \psi^\prime(x_n+)\end{pmatrix}$ satisfy the estimate 
\begin{equation}
\label{eq:estlem1}
\|\xi_n\|_{\dC^2} \ge c_\lambda\frac{\|\xi_0\|_{\dC_2}}{A_n(\lambda)},\quad n\in\dN,
\end{equation}
with some constant $c_\lambda >0$.
\end{lemma} 
\begin{proof}
The sequence of vectors  $\{\xi_n\}_{n\in\dN_0}$ is a solution of the discrete linear system
\begin{equation}
\label{eq:dls}
\xi_n = \Lambda_n\xi_{n-1},\quad n\in\dN,
\end{equation}
with the sequence of matrices $\{\Lambda_n\}_{n\in\dN}$,  having the explicit form
\begin{equation}
\label{eq:Lambda1}
\Lambda_n = \underbrace{\begin{pmatrix} 1 &0  \\ 
\alpha_n &1 \end{pmatrix}}_{J_\delta(\alpha_n)} M_\lambda(\Delta x_{n-1}),
\end{equation}
where $M_\lambda(d)$ is the fundamental matrix given in~\eqref{eq:md} and $J_{\delta}(\alpha)$ is the $\delta$-jump matrix.

One can do the substitution in the discrete linear system~\eqref{eq:dls} of the type
\begin{equation}
\label{eq:subst}
\widetilde\xi_n = \underbrace{\begin{pmatrix} \frac{1}{2} & -\frac{i}{2\sqrt{\lambda}}\\
\frac{1}{2}&\frac{i}{2\sqrt{\lambda}}\end{pmatrix}}_{U_\lambda^{-1}}\xi_n.
\end{equation}
The sequence $\{\widetilde \xi_n\}_{n\in\dN_0}$ is a solution of a new discrete linear system
\begin{equation}
\label{eq:dls2}
\widetilde\xi_{n} = \widetilde\Lambda_n\widetilde\xi_{n-1},\quad n\in\dN,
\end{equation}
where the matrices $\widetilde\Lambda_n$ can be expressed in the following way
\begin{equation*}
\widetilde\Lambda_n = U_\lambda^{-1}J_\delta(\alpha_n) M_\lambda(\Delta x_{n-1}) U_\lambda.
\end{equation*}
Using that $(M_\lambda(d))^{-1} = M_\lambda(-d)$ and $(J_\delta(\alpha))^{-1} = J_\delta(-\alpha)$ we get
\begin{equation}
\label{eq:wtLambda1}
\widetilde\Lambda_n^{-1} = U_\lambda^{-1} M_\lambda(-\Delta x_{n-1})J_\delta(-\alpha_n)  U_\lambda.
\end{equation}
Substituting in~\eqref{eq:wtLambda1} matrices $M_\lambda(d)$, $J_\delta(\alpha)$ and $U_\lambda$ for their expressions given in~\eqref{eq:md},~\eqref{eq:Lambda1} and in~\eqref{eq:subst}, respectively, we get after simple calculations
\begin{equation}
\label{eq:wt1}
\widetilde \Lambda^{-1}_n =  \begin{pmatrix} e^{-i\sqrt\lambda\Delta x_{n-1}} &0\\
0& e^{i\sqrt\lambda\Delta x_{n-1}}\end{pmatrix}\Biggr( \begin{pmatrix} 1&0\\0&1\end{pmatrix} + \frac{i\alpha_n }{2\sqrt{\lambda}}\begin{pmatrix} 1&1\\-1&-1\end{pmatrix}\Biggl).
\end{equation}
Now it is clear that
\begin{equation}
\label{eq:Lambda}
\|\widetilde \Lambda_n^{-1}\| \le  1 + \frac{|\alpha_n|}{\sqrt{\lambda}}.
\end{equation}
From~\eqref{eq:dls2} and~\eqref{eq:Lambda} we get
\begin{equation}
\label{eq:wtest}
\|\widetilde\xi_{n}\|_{\dC^2} \ge \frac{\|\widetilde\xi_{n-1}\|_{\dC^2}}{\|\widetilde\Lambda_n^{-1}\|} \ge  \frac{\|\widetilde\xi_{n-1}\|_{\dC^2}}{1 + \frac{|\alpha_n|}{\sqrt{\lambda}}}.
\end{equation}
The estimate~\eqref{eq:wtest} gives 
\begin{equation}
\label{eq:wtest3}
\|\widetilde\xi_n\|_{\dC^2} \ge \frac{\|\widetilde\xi_{n-1}\|_{\dC^2}}{1+\frac{|\alpha_n|}{\sqrt{\lambda}}}\ge
\frac{\|\widetilde\xi_{n-2}\|_{\dC^2}}{\bigl(1+\frac{|\alpha_{n-1}|}{\sqrt{\lambda}}\bigr)
\bigl(1+\frac{|\alpha_n|}{\sqrt{\lambda}}\bigr)}
\ge \dots \ge \frac{\|\widetilde \xi_0\|_{\dC^2}}{A_n(\lambda)}.
\end{equation}
Returning from $\widetilde \xi$ to $\xi$ we obtain
\begin{equation}
\label{eq:U}
\|\widetilde \xi_n\|_{\dC^2} \le \|\xi_n\|_{\dC^2}\|U_\lambda^{-1}\|,\quad 
\|\widetilde \xi_0\|_{\dC^2} \ge \frac{\|\xi_0\|_{\dC^2}}{\|U_\lambda\|}.
\end{equation}
Putting~\eqref{eq:U} into~\eqref{eq:wtest3},  we get the claim~\eqref{eq:estlem1} with $c_\lambda = \bigl(\|U_\lambda\|\|U_\lambda^{-1}\|\bigr)^{-1}$.
\end{proof}

\begin{lemma}
\label{lem2}
Let $X = \{x_n\}_{n\in\dN_0}$ be a discrete set on the half-line such that~\eqref{eq:cond} holds.
Let $\alpha = \{\alpha_n\}_{n\in\dN}$ be a sequence of real numbers.
Let a function  $\psi \in \cS_{\delta^\prime,X,\alpha}$ be such that~\eqref{eq:diff} holds for some $\lambda>0$.
Let the sequence $\{A_n(\lambda)\}_{n\in\dN_0}$ be defined as in~\eqref{eq:AB}.
Then norms of vectors $\xi_n :=\begin{pmatrix} \psi(x_n+) \\ \psi^\prime(x_n+)\end{pmatrix}$ satisfy the estimate 
\begin{equation}
\label{eq:lemest2}
\|\xi_n\|_{\dC^2} \ge  c_\lambda\frac{\|\xi_0\|_{\dC_2}}{A_n(1/\lambda)},\quad n\in\dN,
\end{equation}
with some constant $c_\lambda >0$.
\end{lemma} 
\begin{proof}
The proof of this lemma is almost the same as the proof of the previous lemma. One should substitute the $\delta$-jump matrix $J_{\delta}(\alpha)$ in~\eqref{eq:Lambda1} for the $\delta^\prime$-jump matrix $J_{\delta^\prime}(\alpha) = \begin{pmatrix} 1 &\alpha \\ 0&1\end{pmatrix}$.  
Repeating the calculations of the previous lemma we get
\begin{equation*}
\widetilde \Lambda_n^{-1} = \begin{pmatrix} e^{-i\sqrt\lambda\Delta x_{n-1}} &0\\
0& e^{i\sqrt\lambda\Delta x_{n-1}}\end{pmatrix}\Biggl(\begin{pmatrix} 1&0\\0&1\end{pmatrix} + \frac{i\alpha_n\sqrt{\lambda} }{2}\begin{pmatrix} -1&1\\-1&1\end{pmatrix}\Biggr).
\end{equation*}
Now it is clear that
\begin{equation*}
\|\widetilde \Lambda_n^{-1}\|\le 1+ |\alpha_n|\sqrt{\lambda}.
\end{equation*}
Analogously to the previous lemma we get the claim~\eqref{eq:lemest2} with $c_\lambda = \bigl(\|U_\lambda\|\|U_\lambda^{-1}\|\bigr)^{-1}$.
\end{proof}

Further we prove two theorems, which contain sufficient conditions on $X$ and $\alpha$   for a subinterval of $\dR_+$ to be free of point spectra of operators $H_{\delta,X,\alpha}$  and $H_{\delta^\prime,X,\alpha}$. 
\begin{theorem}
\label{thm1}
Let $X = \{x_n\}_{n\in\dN_0}$ be a discrete set such that ~\eqref{eq:cond} holds. Let $\alpha = \{\alpha_n\}_{n\in\dN}$ be a sequence of real numbers. Let the self-adjoint operator $H_{\delta,X,\alpha}$ be defined as in~\eqref{eq:opdelta}. Let the sequence  $\{A_n(\lambda)\}_{n\in\dN_0}$ be defined as in~\eqref{eq:AB}. 
If for some $\lambda_0 > 0$ 
\begin{equation}
\label{eq:series1}
\sum_{n=0}^\infty \frac{\Delta x_n}{A_n(\lambda_{0})^2} =  \infty,
\end{equation} 
then the point spectrum of $H_{\delta,X,\alpha}$ satisfies
\begin{equation}
\sigma_{\rm p}\cap \dR_+\subset [0,\lambda_0).
\end{equation}
\end{theorem}
\begin{proof}
Let  $\psi$  be a non-trivial function from the class $\cS_{\delta,X,\alpha}$ such that~\eqref{eq:diff} holds for some $\lambda\ge\lambda_0$.
Let us introduce a sequence $\xi_n = \begin{pmatrix} \psi(x_n+)\\ \psi^\prime(x_n+)\end{pmatrix}$.
According to Lemma~\ref{lem1}
\begin{equation}
\|\xi_n\|_{\dC^2} \ge c_{\lambda} \frac{\|\xi_0\|_{\dC^2}}{A_n(\lambda)}.
\end{equation}
Functions $A_n(\lambda)$ are monotonously decreasing in $\lambda$ for all $n\in\dN$. Hence the divergence of the series in~\eqref{eq:series1} implies 
\begin{equation}
\sum_{n=0}^\infty \|\xi_n\|_{\dC^2}^2 \Delta x_n \ge c_\lambda\|\xi_0\|^2_{\dC^2} \sum_{n=0}^\infty \frac{\Delta x_n}{A_n(\lambda)^2} \ge c_\lambda\|\xi_0\|^2_{\dC^2} \sum_{n=0}^\infty \frac{\Delta x_n}{A_n(\lambda_0)^2} =\infty.
\end{equation}
Then according to Lemma~\ref{lem0} we get $\psi \notin L^2(\dR_+)$ and hence $\lambda\notin\sigma_{\rm p}(H_{\delta,X,\alpha})$.
\end{proof}

\begin{corollary}
\label{cor1}
If we are in the conditions of Theorem~\ref{thm1} and if the sequence $\alpha$ contains only positive real numbers, then the point spectrum of $H_{\delta,X,\alpha}$ satisfies
\begin{equation}
\sigma_{\rm p} \subset \bigr(0,\lambda_0\bigl).
\end{equation}
\end{corollary}
\begin{proof}
We need only to show that there are no eigenvalues in $\dR_-$. Let $\psi$ be an arbitrary function from $\dom\bigl(H_{\delta,X,\alpha}\bigr)$. The scalar product $\bigl(H_{\delta,X,\alpha} \psi,\psi\bigr)_{L^2(\dR_+)}$ can be rewritten, according to the boundary conditions~\eqref{eq:delta}, in the form
\begin{equation}
\label{eq:qdelta}
\|\psi^\prime\|_{L^2(\dR_+)}^2 + \sum_{n\in\dN}\alpha_n|\psi(x_n)|^2.
\end{equation}
Hence $H_{\delta,X,\alpha}\ge 0$ and therefore $\sigma(H_{\delta,X,\alpha})\cap\bigl(-\infty,0\bigr) =\varnothing$.
If $\psi\in\dom\bigl(H_{\delta,X,\alpha}\bigr)$ is such that $H_{\delta,X,\alpha}\psi =0$, then according to~\eqref{eq:qdelta}, $\psi^\prime(x) =0$ on $\dR_+\setminus X$ and $\psi(x_n) = 0$ for all $n\in\dN$, i.~e. the function  $\psi$ is a constant on each interval $(x_n,x_{n+1}),~n\in\dN_0,$ and it takes the value zero at the points $x_n$ for all $n\in\dN$. Therefore $\psi(x) \equiv 0$ and hence $0\notin\sigma_{\rm p}(H_{\delta,X,\alpha})$.
\end{proof}

\begin{theorem}
\label{thm2}
Let $X = \{x_n\}_{n\in\dN_0}$ be a discrete set such that \eqref{eq:cond} holds. Let $\alpha = \{\alpha_n\}_{n\in\dN}$ be a sequence of real numbers. Let the self-adjoint operator $H_{\delta^\prime,X,\alpha}$ be defined as in~\eqref{eq:opdelta'}. Let the sequence $\{A_n(\lambda)\}_{n\in\dN_0}$ be defined as in~\eqref{eq:AB}. 
If for some $\lambda_0 > 0$ 
\begin{equation}
\label{eq:series2}
\sum_{n=0}^\infty \frac{\Delta x_n}{A_n(1/\lambda_0)^2} =  \infty,
\end{equation} 
then the point spectrum of the operator $H_{\delta^\prime, X,\alpha}$ satisfies
\begin{equation*}
\sigma_{\rm p} \cap\dR_+ \subset \bigl(\lambda_0,+\infty\bigr).
\end{equation*}
\end{theorem}
\begin{proof}
The proof of this theorem repeats the proof of Theorem~\ref{thm1} with the only one difference: functions $A_n(1/\lambda)$ are monotonously increasing in $\lambda$ for all $n\in\dN$.
\end{proof}

\begin{corollary}
\label{cor2}
If we are in the conditions of Theorem~\ref{thm2} and if the sequence $\alpha$ contains only positive real numbers, then the point spectrum of  $H_{\delta^\prime, X,\alpha}$ satisfies
\begin{equation*}
\sigma_{\rm p} \subset \bigl(\lambda_0,+\infty\bigr).
\end{equation*}
\end{corollary}
\begin{proof}
The proof is analogous to the proof of Corollary~\ref{cor1}.
\end{proof}

%---------------------------------------------------------------------------------
%---------------------------------------------------------------------------------
\section{Sufficient conditions for $\sigma_{\rm sc}\neq \varnothing$ and for $\sigma=\sigma_{\rm sc} = \dR_+$}
%---------------------------------------------------------------------------------
%---------------------------------------------------------------------------------
\label{sec:4}
In this section we give sufficient conditions on $X$ and $\alpha$ for the operators $H_{\delta,X,\alpha}$ and $H_{\delta^\prime,X,\alpha}$ to have non-empty singular continuous spectra and to have even purely singular continuous spectra. Finally, we give the proof of Theorem~\ref{thm:main} formulated in the introduction. 
We use the results of Section~\ref{sec:3}, the compact perturbation argument and some of the results of  Christ and Stolz~\cite{CS94} and Mikhailets~\cite{M95,M96}. 
\begin{lemma}
\label{lem3}
Let  $X = \{x_n\}_{n\in\dN_0}$ be a discrete set on the half-line such that $\Delta x_n\rightarrow +\infty$. Let $H_{l,\rm D}$ be the one-dimensional Laplacian on the interval of a length $l>0$ with Dirichlet boundary conditions. Let $H_{l,\rm N}$ be one-dimensional Laplacian on the interval of a length $l>0$ with Neumann boundary conditions. Let the self-adjoint operators $H_{X,\rm D}$ and $H_{X,\rm N}$ be defined as direct sums: 
\begin{equation*}
H_{X,\rm D} = \bigoplus_{n=0}^\infty H_{\Delta x_n, \rm D}\quad\text{and}\quad H_{X,\rm N} = \bigoplus_{n=0}^\infty H_{\Delta x_n,\rm N}.
\end{equation*}  
Then the essential spectra of both operators $H_{X,\rm D}$ and $H_{X,\rm N}$ coincide with $\dR_+$.  
\end{lemma}
\begin{proof}
Let us prove the claim only for the operator $H_{X,\rm D}$. The proof for $H_{X,\rm N}$ is analogous.
The operator $H_{X,\rm D}$ is positive as the direct sum of positive operators. 
Let $s > 0$ be an arbitrary positive real number. Let us consider the sequence \begin{equation*}
\lambda_{s,n} = \Biggl(\frac{\pi \bigl\lceil \sqrt{s}\frac{\Delta x_n}{\pi}\bigr\rceil}{\Delta x_n}\Biggr)^2, 
\quad n\in\dN,
\end{equation*}
where $\lceil \cdot\rceil$ is the ceiling function.
Since $\lambda_{s,n} \in \sigma_{\rm p}(H_{\Delta x_n,\rm D})$, then by the definition of $H_{X,\rm D}$ we get $\lambda_{s,n}\in\sigma_{\rm p}(H_{X,\rm D})$. The claim for $H_{X,\rm D}$ follows from the fact that
\begin{equation*}
\lim_{n\rightarrow\infty} \lambda_{s,n} = s.
\end{equation*}
\end{proof} 

The proof of the following lemma is the main step toward the proof of the main result given in Theorem~\ref{thm:main}.
\begin{lemma}
\label{lem4}
Let $X = \{x_n\}_{n\in\dN_0}$ be a discrete set on the half-line such that~\eqref{eq:sparse} holds ($X$ is a sparse set). Let $\alpha = \{\alpha_n\}_{n\in\dN}$  be a sequence of real numbers such that $\alpha_n\rightarrow \infty$. Let the self-adjoint operators $H_{\delta,X,\alpha}$ and $H_{\delta^\prime,X,\alpha}$ be defined as in~\eqref{eq:opdelta} and as in~\eqref{eq:opdelta'}, respectively. Then the following assertions hold:
\begin{itemize}
\item[(i)] if for some $\lambda_0 >0$
\begin{equation}
\label{eq:ser1}
\sum_{n=0}^\infty \frac{\Delta x_n}{\prod_{i=1}^n \Bigl(1+\frac{|\alpha_i|}{\sqrt{\lambda_0}}\Bigr)^2}=\infty,
\end{equation}
then the spectrum of $H_{\delta,X,\alpha}$ has the following structure:
\begin{itemize}
\item[(ess)] $\sigma_{\rm ess}= \bigl[0,+\infty\bigr)$,
\item[(pp)] $\sigma_{\rm pp}\cap\dR_+ \subseteq [0,\lambda_0]$, 
\item[(sc)] $\bigl[\lambda_0,+\infty\bigr)\subseteq\sigma_{\rm sc} \subseteq \bigl[0,+\infty\bigr)$, 
\item[(ac)] $\sigma_{\rm ac} =\varnothing$;
\end{itemize}
 
\item[(ii)] 
if for some $\lambda_0 >0$
\begin{equation}
\label{eq:ser2} 
\sum_{n=0}^\infty \frac{\Delta x_n}{\prod_{i=1}^n \Bigl(1+|\alpha_i|\sqrt{\lambda_0}\Bigr)^2}=\infty,
\end{equation}
then the spectrum of $H_{\delta^\prime,X,\alpha}$ has the structure:
\begin{itemize}
\item[(ess)] $\sigma_{\rm ess} = \bigl[0,+\infty\bigr)$,
\item[(pp)] $\sigma_{\rm pp}\cap\dR_+ \subseteq \bigl[\lambda_0,+\infty\bigr)$, 
\item[(sc)] $\bigl[0,\lambda_0\bigr]\subseteq\sigma_{\rm sc} \subseteq \bigl[0,+\infty\bigr)$, 
\item[(ac)] $\sigma_{\rm ac} =\varnothing$.
\end{itemize}
\end{itemize}
 \end{lemma} 
\begin{proof}

(i) Since $\alpha_n\rightarrow\infty$, then we have according to~\cite[Theorem 3]{CS94} that  $\sigma_{\rm ac}(H_{\delta,X,\alpha}) = \varnothing$. Further, according to~\cite[Theorem 1]{M95}
\begin{equation*}
(H_{\delta,X,\alpha}- \mu) - (H_{X,\rm D} -\mu)^{-1}\in\sS_\infty
\end{equation*}
for all $\mu\in\rho(H_{\delta,X,\alpha})\cap\rho(H_{X,\rm D})$. Therefore by the compact perturbation argument and Lemma~\ref{lem3}
\begin{equation}
\label{eq:ess}
\sigma_{\rm ess}(H_{\delta,X,\alpha}) = \sigma_{\rm ess}(H_{X,\rm D}) = \dR_+.
\end{equation}
By Theorem~\ref{thm1} 
\begin{equation}
\label{eq:pp}
\sigma_{\rm pp}(H_{\delta,X,\alpha})\cap \dR_+ \subseteq \bigl[0,\lambda_0\bigr].
\end{equation}
Taking into account the emptiness of the absolutely continuous spectrum  we get from~\eqref{eq:ess} that
\begin{equation*}
\bigl(\sigma_{\rm sc}\cup\sigma_{\rm pp}\bigr)(H_{\delta,X,\alpha}) \supseteq \sigma_{\rm ess}(H_{\delta,X,\alpha}) = \dR_+.
\end{equation*}
Then according to~\eqref{eq:pp} we obtain
\begin{equation*}
\bigl[\lambda_0,+\infty\bigr)\subseteq\sigma_{\rm sc}(H_{\delta,X,\alpha})\subseteq\bigl[0,+\infty\bigr).
\end{equation*}

(ii) The idea is similar to the one used in the proof of the item~(i). In order to prove that $\sigma_{\rm ac}(H_{\delta^\prime,X,\alpha}) = \varnothing$ one should make some minor changes in~\cite[Theorem 3]{CS94}, see also \cite[Theorem 1]{M96}. According to~\cite[Theorem 1]{M95} the operator $H_{\delta^\prime,X,\alpha}$ is a compact perturbation of the operator $H_{X,\rm N}$ in the resolvent difference sense. Hence by the compact perturbation argument and Lemma~\ref{lem3} 
\begin{equation}
\label{eq:ess2}
\sigma_{\rm ess}(H_{\delta^\prime,X,\alpha}) = \sigma_{\rm ess}(H_{X,\rm N}) = \dR_+.
\end{equation}
By Theorem~\ref{thm2} 
\begin{equation}
\label{eq:pp2}
\sigma_{\rm pp}(H_{\delta^\prime,X,\alpha})\cap\dR_+ \subseteq \bigl[\lambda_0,+\infty\bigr).
\end{equation}
Again taking into account that the absolutely continuous spectrum is empty we get from~\eqref{eq:ess2} that
\begin{equation*}
\bigl(\sigma_{\rm sc}\cup\sigma_{\rm pp}\bigr)(H_{\delta^\prime,X,\alpha}) \supseteq \sigma_{\rm ess}(H_{\delta^\prime,X,\alpha}) = \dR_+.
\end{equation*}
Then according to~\eqref{eq:pp2} we get
\begin{equation*}
\bigl[0,\lambda_0\bigr]\subseteq\sigma_{\rm sc}(H_{\delta^\prime,X,\alpha})\subseteq\bigl[0,+\infty\bigr).
\end{equation*}

\end{proof}
\begin{corollary}
\label{cor3}
If we are in the conditions of Lemma~\ref{lem4} and if the sequence $\alpha$ contains only positive real numbers, then the following assertions hold:
\begin{itemize}
\item[(i)] if the series in~\eqref{eq:ser1} diverges for all $\lambda_0 >0$, then
the spectrum of $H_{\delta,X,\alpha}$ is purely singular continuous and coincides with $\dR_+$;
\item[(ii)] if the series in~\eqref{eq:ser2} diverges for all $\lambda_0 >0$,
then the spectrum of $H_{\delta^\prime,X,\alpha}$ is purely singular continuous and coincides with $\dR_+$.
\end{itemize}
\end{corollary}
\begin{proof}
The item~(i) follows from Lemma~\ref{lem4}~(i) and Corollary~\ref{cor1}. The item~(ii) follows from Lemma~\ref{lem4}~(ii) and Corollary~\ref{cor2}.
\end{proof}
\begin{remark}
The conditions in the items (i) and (ii) of Corollary~\ref{cor3} are indeed equivalent.
\end{remark}
Further we give the proof of our main result.
%---------------------------------------------------------------------------------
\subsection*{Proof of Theorem~\ref{thm:main}}
Recall from the introduction that we define the value $a\in\dR_+\cup\{+\infty\}$ as the following limit
\begin{equation*}
a:= \liminf_{n\rightarrow\infty} \frac{\Delta x_n}{\Delta x_{n-1}\alpha_n^2}.
\end{equation*}  
Let us apply d'Alembert principle to the series~\eqref{eq:ser1} from Lemma~\ref{lem4}~(i). For the divergence of this series it is sufficient to satisfy the condition
\begin{equation}
\liminf_{n\rightarrow\infty} \frac{\Delta x_n}{\Delta x_{n-1}\Bigl(1 +\frac{2}{\sqrt\lambda_0}|\alpha_{n}| + \frac{1}{\lambda_0}\alpha_{n}^2\Bigr)} = \liminf_{n\rightarrow\infty} \lambda_0\frac{\Delta x_{n}}{\Delta x_{n-1}\alpha_{n}^2} = \lambda_0 a > 1.
\end{equation}
If $0 < a< +\infty$, then for all $\lambda_0 > \frac{1}{a}$ the series~\eqref{eq:ser1} diverges and we get the item~(i-a). 

Analogously, let us apply d'Alembert principle to the series~\eqref{eq:ser2} from Lemma~\ref{lem4}~(ii). For the divergence of this series it is sufficient to satisfy the condition
\begin{equation}
\liminf_{n\rightarrow\infty} \frac{\Delta x_n}{\Delta x_{n-1}\Bigl(1 +2\sqrt\lambda_0|\alpha_{n}| + \lambda_0\alpha_{n}^2\Bigr)} = \liminf_{n\rightarrow\infty} \frac{1}{\lambda_0}\frac{\Delta x_{n}}{\Delta x_{n-1}\alpha_{n}^2} = \frac{a}{\lambda_0} > 1.
\end{equation}
If $0 < a <+\infty$, then for all $\lambda_0 \in \bigl(0,a\bigr)$ the series~\eqref{eq:ser2} diverges and we get the item~(i-b). 

If $a = +\infty$, then according to d'Alembert principle both series in~\eqref{eq:ser1} and in~\eqref{eq:ser2} diverge for all $\lambda_0 >0$ and we get from Corollary~\ref{cor3} the item~(ii).

\section{Discussion based on examples}

Consider the Schr\"odinger operator formally given by the expression
\begin{equation}
-\frac{d^2}{dx^2} + \sum_{n\in\dN} n^{1/4}\delta_{n!}.
\end{equation}
Since
\begin{equation}
\liminf_{n\rightarrow\infty} \frac{n\cdot n!}{(n-1)\cdot(n-1)!n^{1/2}} = \liminf_{n\rightarrow\infty} \frac{n^2}{n^{3/2}-n^{1/2}}=+\infty,
\end{equation}
then by Theorem~\ref{thm:main}~(ii) the spectrum is purely singular continuous and coincides with $\dR_+$.

Another example is the Schr\"odinger operator formally given by the expression
\begin{equation}
-\frac{d^2}{dx^2} + \sum_{n\in\dN} n^{1/2}\delta_{n!}.
\end{equation}
Since
\begin{equation}
\liminf_{n\rightarrow\infty} \frac{n\cdot n!}{(n-1)\cdot(n-1)!n} = \liminf_{n\rightarrow\infty} \frac{n}{n-1}=1,
\end{equation}
then by Theorem~\ref{thm:main}~(i) the singular continuous spectrum is non-empty and contains the interval $[1,+\infty)$

Our results do not allow to define exactly the structure of the spectrum  on the interval $[0,1]$ in the last example. The spectrum on this interval may be purely singular continuous, only pure point or a mixture of these two kinds of spectra.  
It is a question of interest for the author to construct an operator $H_{\delta,X,\alpha}$ with $\delta$-interactions on a sparse discrete set $X$, having non-empty positive singular continuous spectrum and non-empty positive pure point spectrum, or to establish that this situation can not occur. Another question of interest is to estimate the Hausdorff dimension of the obtained singular continuous spectrum. 
The last question will be discussed in the subsequent publication.

\section*{Acknowledgments}
\label{ack}
The author would like to thank Dr.~S.~Simonov for careful reading of the manuscript, valuable remarks and suggestions. The author also would like to express his gratitude to Dr.~A.~Kostenko and Prof.~I.~Yu.~Popov for discussions. The work was supported by the grant 2.1.1/4215 of the program ''Development of the potential of the high school in Russian Federation 2009-2010'' and by the grant  NK-526P/24 of the program ''Scientific stuff of innovative Russia".

%----------------------------------------------------------------------
\newcommand{\etalchar}[1]{$^{#1}$}
\providecommand{\bysame}{\leavevmode\hbox to3em{\hrulefill}\thinspace}
\providecommand{\MR}{\relax\ifhmode\unskip\space\fi MR }
\renewcommand{\MR}[1]{}  % don't want MR-numbers in list

{\bf Vladimir Lotoreichik\\
\\
{\small Department of Mathematics}\\
{\small St.~Petersburg State University of IT, Mechanics and Optics}\\ 
{\small 197101, St.~Petersburg, Kronverkskiy pr., d.~49}\\
{\small E-mail: vladimir.lotoreichik@gmail.com}
}

\end{document}